\documentclass[11pt,a4paper]{article}
\pdfoutput=1
\usepackage[dvips]{graphicx}
\usepackage{subfig}
\usepackage[pagewise]{lineno}
\usepackage{xcolor}
\usepackage[english]{babel}
\usepackage{tabularx}
\usepackage[T1]{fontenc}
\usepackage[utf8]{inputenc}
\usepackage{lmodern}
\usepackage{hyperref}
\usepackage{calc}
\usepackage{tikz}
\usetikzlibrary{arrows}
\usepackage{verbatim}
\usepackage{mathtools}
\usepackage{amsmath,amsfonts,amssymb}
\usepackage{url}
\usepackage{cite}
\usepackage{breqn}
\usepackage[english]{babel}
\usepackage{amsthm}
\usepackage{latexsym}
\usepackage[mathscr]{euscript}
\newtheorem{definition}{Definition}[section]
\theoremstyle{definition}
\newtheorem{example}{Example}[section]
\newtheorem{theorem}{Theorem}[section]
\usepackage{authblk}
\usepackage{ifpdf}
\ifpdf
  \usepackage{epstopdf}
\fi
\begin{document}
\title{Determination of order in linear fractional differential equations}
\author{Mirko D'Ovidio,\ Paola Loreti,\ Alireza Momenzadeh and Sima Sarv Ahrabi}
\maketitle
\begin{abstract}
In this article, the order of some classes of fractional linear differential equations is determined, based on asymptotic behaviour of the solution as time tends to infinity. The order of fractional derivative has been proved to be of great importance in an accurately appropriate simulation of the system under study. Specifically, by representing the asymptotic expansion of the solution, it could be obviously demonstrated that the decay rate of the solution is influenced by the order of fractional differentiation. The numerical investigation is conducted into the proven formulae.
\end{abstract}
\section{Introduction}\label{sec:1}
\setcounter{section}{1}
The practical significance of fractional calculus has been recently discerned as a vastly superior method of describing the long-memory processes and had a remarkable development over the last few years, both in mathematical and non-mathematical fields \cite{baleanu2010new,diethelm2010analysis,kilbas2006theory,mainardi2010fractional,sabatier2007advances,caputo2015new}. More specifically, fractional differential equations (FDEs) have been proven extremely important for more accurately modelling of many physical phenomena \cite{carpinteri2014fractals,rossikhin2010application,baleanu2009central,magin2010fractional,kamocki2014pontryagin}.
Inverse problems to FDEs occur in many branches of science. Such problems have been investigated, for instance, in fractional diffusion equation \cite{cheng2009uniqueness,hatano2013determination,lopushansky2014one,sakamoto2011initial,zhang2011inverse} and inverse boundary value problem for semi-linear fractional telegraph equation \cite{lopushanska2015inverse}. Specifically in \cite{hatano2013determination}, it has been demonstrated that determination of $\beta $, the order of fractional differential operator, is definitely crucial to the appropriate simulation of the anomalous diffusion in order to specify that the transport phenomenon exhibits sub-diffusion or super-diffusion for respectively $\beta <1$ and $\beta >1$. The authors in \cite{hatano2013determination} have presented and proven a theorem, the idea behind which is seeking solutions to an inverse problem, i.e. determination of the order of a fractional diffusion equation; therefore, this persuaded us to prove formulae indicating the relationship between the fractional order and the asymptotic behaviour of the exact solutions to several different class of fractional differential equations. The following sections are allocated to preliminaries to fractional calculus, the main results and numerical simulation.
\section{Preliminaries}
It is appropriate to briefly recall some critical bases of fractional calculus.
\subsection{Fractional integral and derivatives}
In this part, the definitions of fractional integral and derivative of Riemann-Liouville type and also the fractional derivative of Caputo type are represented in summary \cite{kilbas2006theory,podlubny1998fractional}.
\begin{definition}
\label{def1}
Let ${{t}_{0}}$ be a real number and let the function $f:\left( {{t}_{0}},+\infty  \right)\to \mathbb{R}$ be continuous and integrable in every finite interval $\left( {{t}_{0}},t \right)$. The Riemann-Liouville fractional integral of order $\beta \in \mathbb{C}$ $\left( \Re \left( \beta  \right)>0 \right)$ of the function $f$ is defined by
\begin{equation}
\label{2}
\left( I_{{{t}_{0}}}^{\beta }f \right)\left( t \right)=\frac{1}{\Gamma \left( \beta  \right)}\int_{{{t}_{0}}}^{t}{{{\left( t-\tau  \right)}^{\beta -1}}f\left( \tau  \right)d\tau },\text{ }t>{{t}_{0}}.
\end{equation}
\end{definition}
\begin{definition}
\label{def2}
Let ${{t}_{0}}$ be a real number, $\beta \in \mathbb{C}$ $\left( \Re \left( \beta  \right)>0 \right)$, $n=\left[ \Re \left( \beta  \right) \right]+1$ and let the function $f:\left( {{t}_{0}},+\infty  \right)\to \mathbb{R}$ be continuous and integrable in every finite interval $\left( {{t}_{0}},t \right)$. The Riemann-Liouville fractional derivative of order $\beta $ of the function $f$ for $t>{{t}_{0}}$ is defined by
\begin{equation}
\label{3}
\begin{split}
\left( D_{{{t}_{0}}}^{\beta }f \right)\left( t \right)&={{\left( \frac{d}{dt} \right)}^{n}}\left( I_{{{t}_{0}}}^{n-\beta }f \right)\left( t \right)\\ &=\frac{1}{\Gamma \left( n-\beta  \right)}{{\left( \frac{d}{dt} \right)}^{n}}\int_{{{t}_{0}}}^{t}{{{\left( t-\tau  \right)}^{n-\beta -1}}f\left( \tau  \right)d\tau }
\end{split}
\end{equation}
\end{definition}
The fractional differentiation of Riemann-Liouville type possesses remarkable significance for advancement in the theory of fractional differentiation and integrals and in pure mathematics \cite{podlubny1998fractional} but, nonetheless, the Riemann-Liouville differential operator results in initial conditions incorporating the limit values of the derivative at the lower terminal $t={{t}_{0}}$, for instance $\lim_{ t\rightarrow {t}_{0} }\left( D_{{{t}_{0}}}^{\beta -1}u \right)\left( t \right)={{u}_{1}}.$ As there is no physical interpretation for this type of initial conditions, Caputo differential operator plays a major role in physical phenomena due to the fact that the initial conditions for the FDEs with Caputo derivative are the same as those for integer-order differential equations. Caputo fractional derivative \cite{caputo1967linear} is defined as follows:
\begin{definition}
\label{def3}
Let ${{t}_{0}}\in \mathbb{R}$, $\beta\in\mathbb{C}$ $\left( \Re \left(\beta\right)>0 \right)$, $n-1<\Re\left(\beta\right)<n$ $\left(n\in \mathbb{N} \right)$ and also let the function $f$ be continuous and have $n$ continuous derivatives in the interval $\left({{t}_{0}},+\infty \right)$. The Caputo fractional derivative of order $\beta $ of the function $f$ is defined by
\begin{equation}
\label{4}
\left( D_{t}^{\beta }f \right)\left( t \right)=\frac{1}{\Gamma\left( n-\beta\right)}\int_{{{t}_{0}}}^{t}{{{\left( t-\tau\right)}^{n-\beta -1}}{{f}^{\left( n \right)}}\left(\tau\right)d\tau },\quad { }t>{{t}_{0}}.
\end{equation}
\end{definition}
Furthermore, sequential fractional derivative (see \cite{kilbas2006theory,podlubny1998fractional}) is defined by
\begin{equation}
\label{sequential}
{{\mathscr{D}}^{n\beta }}u\left( t \right)=\underbrace{{{D}^{\beta }}{{D}^{\beta }}...{{D}^{\beta }}u\left( t \right)}_{n}
\end{equation}
where ${{D}^{\beta }}$ could be Riemann-Liouville, Caputo or any other type of fractional derivative not considered here.
\subsection{Mittag-Leffler function and its derivatives}
The generalization of exponential function, ${{e}^{z}}$, was introduced by Mittag-Leffler and is denoted by
\begin{equation}
\label{5}
{{E}_{\alpha}}\left( z \right)=\sum\limits_{k=0}^{\infty }{\frac{{{z}^{k}}}{\Gamma \left( k\alpha +1\right)}},\quad \alpha ,z\in \mathbb{C},\quad \Re \left( \alpha  \right)>0.
\end{equation}
The two-parameter function of Mittag-Leffler type, which first appeared in an article by Wiman \cite{wiman1905fundamentalsatz} and studied by Agarwal and Humbert \cite{agarwal1953propos,humbert1953quelques}, is defined by
\begin{equation}
\label{6}
{{E}_{\alpha ,\beta }}\left( z \right)=\sum\limits_{k=0}^{\infty }{\frac{{{z}^{k}}}{\Gamma \left( k\alpha +\beta  \right)}},\quad \alpha ,\beta ,z\in \mathbb{C},\quad \Re \left( \alpha  \right)>0.
\end{equation}
In the case $\alpha $ and $\beta $ are real and positive numbers, the series converges for all values of the argument $z$, therefore, the Mittag-Leffler function ${{E}_{\alpha ,\beta }}\left( z \right)$ is an entire function of the order ${{\alpha }^{-1}}$.
The Mittag-Leffler function satisfies the recurrence property
\begin{equation}
\label{9}
{{E}_{\alpha ,\beta }}\left( z \right)=-\frac{1}{z\Gamma \left( \beta -\alpha  \right)}+\frac{1}{z}{{E}_{\alpha ,\beta -\alpha }}\left( z \right).
\end{equation}
By the fractional differentiation operator of the Riemann-Liouville type $D_{0}^{\gamma }f$ ($\gamma \in \mathbb{R}$), the Mittag-Leffler function satisfies the following differentiation formula \cite{podlubny1998fractional}
\begin{equation}
\label{10}
D_{0}^{\gamma }\left( {{z}^{\alpha k+\beta -1}}E_{\alpha ,\beta }^{\left( k \right)}\left( \lambda {{z}^{\alpha }} \right) \right)={{z}^{\alpha k+\beta -\gamma -1}}E_{\alpha ,\beta -\gamma }^{\left( k \right)}\left( \lambda {{z}^{\alpha }} \right).
\end{equation}
The particular case of the relationship (\ref{10}) for $n\in \mathbb{N}$ has the form below
\begin{equation}
\label{11}
{{\left( \frac{d}{dz} \right)}^{n}}\left( {{z}^{\beta -1}}{{E}_{\alpha ,\beta }}\left( \lambda {{z}^{\alpha }} \right) \right)={{z}^{\beta -n-1}}{{E}_{\alpha ,\beta -n}}\left( \lambda {{z}^{\alpha }} \right),
\end{equation}
and the following practical formulae could be directly derived from (\ref{11}):
\begin{equation}
\label{12}
\frac{d}{dz}{{E}_{\alpha }}\left( \lambda {{z}^{\alpha }} \right)=\frac{1}{z}{{E}_{\alpha ,0}}\left( \lambda {{z}^{\alpha }} \right)=\lambda {{z}^{\alpha -1}}{{E}_{\alpha ,\alpha }}\left( \lambda {{z}^{\alpha }} \right),
\end{equation}
\begin{equation}
\label{13}
\frac{d}{dz}\left( {{E}_{\alpha ,\alpha }}\left( \lambda {{z}^{\alpha }} \right) \right)=\frac{1}{z}{{E}_{\alpha ,\alpha -1}}\left( \lambda {{z}^{\alpha }} \right)+\frac{\left( 1-\alpha  \right)}{z}{{E}_{\alpha ,\alpha }}\left( \lambda {{z}^{\alpha }} \right),
\end{equation}
\begin{equation}
\label{14}
\frac{d}{dz}\left( {{z}^{\beta -1}}{{E}_{\alpha ,\beta }}\left( \lambda {{z}^{\alpha }} \right) \right)={{z}^{\beta -2}}{{E}_{\alpha ,\beta -1}}\left( \lambda {{z}^{\alpha }} \right),
\end{equation}
\begin{equation}
\label{15}
\frac{d}{dz}\left( z{{E}_{\alpha ,2}}\left( \lambda {{z}^{\alpha }} \right) \right)={{E}_{\alpha }}\left( \lambda {{z}^{\alpha }} \right).
\end{equation}
\subsection{Asymptotic expansion of Mittag-Leffler function}
The asymptotic behaviour of two-parameter Mittag-Leffler function ${{E}_{\alpha ,\beta }}\left( z \right)$ is complicated for $\alpha >0$. The asymptotic expansion of ${{E}_{\alpha ,\beta }}\left( z \right)$ $\left( \left| z \right|\to \infty\right)$ diverges greatly for $0<\alpha <2$ and $\alpha \ge 2$. In this section, the asymptotic behaviour of Mittag-Leffler function is briefly stated for the case $0<\alpha <2$. Referring to \cite{podlubny1998fractional}, the issue could be perfectly investigated with scrupulous attention to detail.
Suppose that $0<\alpha <2$, $\beta ,z\in \mathbb{C}$ and $\mu$ be an arbitrary real number such that $\frac{\pi }{2}\alpha <\mu <\min \left( \pi ,\pi \alpha  \right)$. Then the following expansions hold
\begin{equation}
\label{19}
\begin{split}
{{E}_{\alpha ,\beta }}\left( z \right)= \frac{1}{\alpha }{{z}^{\frac{1-\beta }{\alpha }}}\exp \left( {{z}^{\frac{1}{\alpha }}} \right)-& \sum\limits_{k=1}^{n}{\frac{1}{\Gamma \left( \beta -k\alpha  \right){{z}^{k}}}+O\left( {{\left| z \right|}^{-n-1}} \right)},\\
\quad &\left| z \right|\to \infty , \quad \left| \arg \left( z \right) \right|\le \mu .
\end{split}
\end{equation}
and
\begin{equation}
\label{20}
\begin{split}
{{E}_{\alpha ,\beta }}\left( z \right)= -\sum\limits_{k=1}^{n}{\frac{1}{\Gamma \left( \beta -k\alpha  \right){{z}^{k}}}}+ &O\left( {{\left| z \right|}^{-n-1}} \right),\\& \quad \left| z \right|\to \infty ,\quad \mu \le \left| \arg \left( z \right) \right|\le \pi .
\end{split}
\end{equation}
By applying the expansion (\ref{20}) to the reals $\left( z\in \mathbb{R} \right)$, the following advantageous formulae could be acquired
\begin{equation}
\label{22}
{{E}_{\alpha }}\left( \lambda {{z}^{\alpha }} \right)=-\frac{{{z}^{-\alpha }}}{\lambda \Gamma \left( 1-\alpha  \right)}+O\left( \frac{1}{{{\left| \lambda \right| }^{2}}{{z}^{2\alpha }}} \right),\quad  z\to \infty ,\quad  z>0,\quad  \lambda <0,
\end{equation}
\begin{equation}
\label{24}
{{E}_{\alpha ,\alpha }}\left( \lambda {{z}^{\alpha }} \right)=\frac{\alpha {{z}^{-2\alpha }}}{{{\lambda }^{2}}\Gamma \left( 1-\alpha  \right)}+ O\left( \frac{1}{{{\left| \lambda \right| }^{3}}{{z}^{3\alpha }}} \right),\quad z\to \infty ,\quad z>0,\quad \lambda <0.
\end{equation}
\section{Main result}
This section is intended to determine the order of several classes of fractional differential equations by using the asymptotic behaviour of the exact solutions, as time tends to infinity.
\begin{theorem}
\label{th1}
Let $0<\beta \le 1$, ${{t}_{0}}>0$ and also let $D_{{t}_{0}}^{\beta }u$ represents the Riemann-Liouville differentiation operator. Consider the sequential linear differential equation of fractional order
\begin{equation}
\label{105}
{{\mathscr{D}}_{{t}_{0}}^{2\beta }}u+{{a}_{1}}{{\mathscr{D}}_{{t}_{0}}^{\beta }}u+{{a}_{0}}u=0,
\end{equation}
with the initial condition $u\left( {{t}_{0}} \right)={{u}_{0}}$ and $\mathscr{D}_{{t}_{0}}^{\beta }u\left( {{t}_{0}} \right)={{u}_{1}}$, and let ${{a}_{0}}$ and ${{a}_{1}}$ are reals such that ${{r}_{1}}$ and ${{r}_{2}}$, the roots of the characteristic equation ${{r}^{2}}+{{a}_{1}}r+{{a}_{0}}=0,$ are distinct and real negative numbers. The following formula holds
\begin{equation}
\beta =-1-\lim_{t\to\infty} \frac{t{u}'}{u}
\end{equation}
\end{theorem}
\begin{proof}
The exact solution to Eq.(\ref{105}) has the form \cite{kilbas2006theory}
\begin{equation}
\label{107}
u\left( t \right)={{c}_{1}}{{t}^{\beta -1}}{{E}_{\beta ,\beta }}\left( {{r}_{1}}{{t}^{\beta }} \right)+{{c}_{2}}{{t}^{\beta -1}}{{E}_{\beta ,\beta }}\left( {{r}_{2}}{{t}^{\beta }} \right).
\end{equation}
where ${{c}_{1}}$ and ${{c}_{2}}$ depend on the initial conditions. The first derivative of $u\left( t \right)$ could be calculated by using (\ref{14}) as below
\begin{equation}
\label{108}
{u}'\left( t \right)={{c}_{1}}{{t}^{\beta -2}}{{E}_{\beta ,\beta -1}}\left( {{r}_{1}}{{t}^{\beta }} \right)+{{c}_{2}}{{t}^{\beta -2}}{{E}_{\beta ,\beta -1}}\left( {{r}_{2}}{{t}^{\beta }} \right).
\end{equation}
By referring to (\ref{20}), the asymptotic expansion of $u\left( t \right)$ and ${u}'\left( t \right)$ are respectively
\begin{equation}
\label{109}
u\left( t \right)=\frac{-{{t}^{-\beta -1}}}{\Gamma \left( -\beta  \right)}\left( \frac{{{c}_{1}}}{r_{1}^{2}}+\frac{{{c}_{2}}}{r_{2}^{2}} \right)+{{c}_{1}}{{t}^{\beta -1}}O\left( {{\left| {{r}_{1}} \right|}^{-3}}{{t}^{-3\beta }} \right)+{{c}_{2}}{{t}^{\beta -1}}O\left( {{\left| {{r}_{2}} \right|}^{-3}}{{t}^{-3\beta }} \right)
\end{equation}
and
\begin{equation}
\label{110}
\begin{split}
{u}'\left( t \right)=\frac{\left( \beta +1 \right){{t}^{-\beta -2}}}{\Gamma \left( -\beta  \right)}\left( \frac{{{c}_{1}}}{r_{1}^{2}}+\frac{{{c}_{2}}}{r_{2}^{2}} \right)&+{{c}_{1}}{{t}^{\beta -2}}O\left( {{\left| {{r}_{1}} \right|}^{-3}}{{t}^{-3\beta }} \right)\\& +{{c}_{2}}{{t}^{\beta -2}}O\left( {{\left| {{r}_{2}} \right|}^{-3}}{{t}^{-3\beta }} \right).
\end{split}
\end{equation}
Therefore
\begin{equation}
\label{111}
\frac{t{u}'}{u}=\frac{\frac{\left( \beta +1 \right){{t}^{-\beta -1}}}{\Gamma \left( -\beta  \right)}\left( \frac{{{c}_{1}}}{r_{1}^{2}}+\frac{{{c}_{2}}}{r_{2}^{2}} \right)+{{c}_{1}}{{t}^{\beta -1}}O\left( {{\left| {{r}_{1}} \right|}^{-3}}{{t}^{-3\beta }} \right)+{{c}_{2}}{{t}^{\beta -1}}O\left( {{\left| {{r}_{2}} \right|}^{-3}}{{t}^{-3\beta }} \right)}{\frac{-{{t}^{-\beta -1}}}{\Gamma \left( -\beta  \right)}\left( \frac{{{c}_{1}}}{r_{1}^{2}}+\frac{{{c}_{2}}}{r_{2}^{2}} \right)+{{c}_{1}}{{t}^{\beta -1}}O\left( {{\left| {{r}_{1}} \right|}^{-3}}{{t}^{-3\beta }} \right)+{{c}_{2}}{{t}^{\beta -1}}O\left( {{\left| {{r}_{2}} \right|}^{-3}}{{t}^{-3\beta }} \right)}
\end{equation}
As $t\to \infty $, Eq. (\ref{111}) leads to
\begin{equation}
\label{112}
-\lim_{t\to\infty}\,\frac{t{u}'}{u}= -\lim_{t\to\infty} \,\frac{\frac{\left( \beta +1 \right){{t}^{-\beta -1}}}{\Gamma \left( -\beta  \right)}\left( \frac{{{c}_{1}}}{r_{1}^{2}}+\frac{{{c}_{2}}}{r_{2}^{2}} \right)}{\frac{-{{t}^{-\beta -1}}}{\Gamma \left( -\beta  \right)}\left( \frac{{{c}_{1}}}{r_{1}^{2}}+\frac{{{c}_{2}}}{r_{2}^{2}} \right)}=\beta +1
\end{equation}
and proof is completed.
\end{proof}
\begin{theorem}
\label{th2}
Let $0<\beta <\frac{1}{2}$, $\gamma ,\mu \in \mathbb{R}$ such that $0<\gamma <{{\mu }^{2}}$, and let $D_{t}^{\beta }u$ indicates the Caputo differentiation operator. For the initial value problem
\begin{equation}
\label{1}
D_{t}^{2\beta }u\left( t \right)+2\mu D_{t}^{\beta }u\left( t \right)+\gamma u\left( t \right)=0
\end{equation}
with the initial condition $u\left( 0 \right)=1$, and also for \it sequential linear differential equation of fractional order
\begin{equation}
\label{104}
{{\mathscr{D}}_{t}^{2\beta }}u+2\mu {{\mathscr{D}}_{t}^{\beta }}u+\gamma u=0
\end{equation}
with the initial condition $\mathscr{D}_{t}^{\beta }u\left( 0 \right)=0$ and $u\left( 0 \right)=1$, the following formula holds
\begin{equation}
\label{29}
\beta = -\lim_{t\to\infty} \frac{t{u}'}{u}
\end{equation}
{\bf Remark}: If $\mathscr{D}_{t}^{\beta }u\left( 0 \right)=0,$ then ${{c}_{1}}{{r}_{1}}+{{c}_{2}}{{r}_{2}}=0$ and $\mathscr{D}_{t}^{2\beta }u=D_{t}^{2\beta }u$. The case of $\mathscr{D}_{t}^{\beta }u\left( 0 \right)\ne 0$ leads to ${{c}_{1}}{{r}_{1}}+{{c}_{2}}{{r}_{2}}\ne 0$ and therefore, the coefficients ${{c}_{1}}$ and ${{c}_{2}}$ are not the same as those represented in the proof of Theorem \ref{th2} and must be calculated.
\end{theorem}
\begin{proof}
The equations (\ref{1}) and (\ref{104}) have the exact solution \cite{d2014time}, represented by
\begin{equation}
\label{28}
u\left( t \right)={{c}_{1}}{{E}_{\beta }}\left( {{r}_{1}}{{t}^{\beta }} \right)+{{c}_{2}}{{E}_{\beta }}\left( {{r}_{2}}{{t}^{\beta }} \right),
\end{equation}
where the coefficients ${{c}_{1}}$ and ${{c}_{2}}$ are respectively equal to $\frac{1}{2}\left( 1+\frac{\mu }{\sqrt{{{\mu }^{2}}-\gamma }} \right)$ and $\frac{1}{2}\left( 1-\frac{\mu }{\sqrt{{{\mu }^{2}}-\gamma }} \right)$, and the parameters ${{r}_{1}}$ and ${{r}_{2}}$ equal to $-\mu +\sqrt{{{\mu }^{2}}-\gamma }<0$ and $-\mu -\sqrt{{{\mu }^{2}}-\gamma }<0$ respectively.
The asymptotic behaviour of Mittag-Leffler function at infinity is applied to (\ref{28}). By using (\ref{22})
\begin{equation}
\label{30}
u\left( t \right)=-\frac{{{t}^{-\beta }}}{\Gamma \left( 1-\beta  \right)}\left( \frac{{{c}_{1}}}{{{r}_{1}}}+\frac{{{c}_{2}}}{{{r}_{2}}} \right)+{{c}_{1}}O\left( r_{1}^{-2}{{t}^{-2\beta }} \right)+{{c}_{2}}O\left( r_{2}^{-2}{{t}^{-2\beta }} \right).
\end{equation}
The first derivative of (\ref{28}) could be obtained by referring to (\ref{12})
\begin{equation}
\label{31}
{u}'\left( t \right)={{c}_{1}}{{r}_{1}}{{t}^{\beta -1}}{{E}_{\beta ,\beta }}\left( {{r}_{1}}{{t}^{\beta }} \right)+{{c}_{2}}{{r}_{2}}{{t}^{\beta -1}}{{E}_{\beta ,\beta }}\left( {{r}_{2}}{{t}^{\beta }} \right),
\end{equation}
and using (\ref{24}) and applying the asymptotic behaviour of Mittag-Leffler function to (\ref{31}), leads to
\begin{equation}
\label{32}
{u}'\left( t \right)=\frac{\beta {{t}^{-\beta -1}}}{\Gamma \left( 1-\beta  \right)}\left( \frac{{{c}_{1}}}{{{r}_{1}}}+\frac{{{c}_{2}}}{{{r}_{2}}} \right)+{{t}^{\beta -1}}\left( {{c}_{1}}{{r}_{1}}O\left( r_{1}^{-3}{{t}^{-3\beta }} \right)+{{c}_{2}}{{r}_{2}}O\left( r_{2}^{-3}{{t}^{-3\beta }} \right) \right).
\end{equation}
Therefore
\begin{equation}
\label{33}
\frac{t{u}'}{u}=t\frac{\frac{\beta {{t}^{-\beta -1}}}{\Gamma \left( 1-\beta  \right)}\left( \frac{{{c}_{1}}}{{{r}_{1}}}+\frac{{{c}_{2}}}{{{r}_{2}}} \right)+{{t}^{\beta -1}}\left( {{c}_{1}}{{r}_{1}}O\left( r_{1}^{-3}{{t}^{-3\beta }} \right)+{{c}_{2}}{{r}_{2}}O\left( r_{2}^{-3}{{t}^{-3\beta }} \right) \right)}{-\frac{{{t}^{-\beta }}}{\Gamma \left( 1-\beta  \right)}\left( \frac{{{c}_{1}}}{{{r}_{1}}}+\frac{{{c}_{2}}}{{{r}_{2}}} \right)+{{c}_{1}}O\left( r_{1}^{-2}{{t}^{-2\beta }} \right)+{{c}_{2}}O\left( r_{2}^{-2}{{t}^{-2\beta }} \right)}.
\end{equation}
As $t\to \infty $, from (\ref{33}) the result could be obtained
\begin{equation}
\label{34}
-\lim_{t\to\infty}\frac{t{u}'}{u}= -\lim_{t\to\infty} t\frac{\frac{\beta {{t}^{-\beta -1}}}{\Gamma \left( 1-\beta  \right)}\left( \frac{{{c}_{1}}}{{{r}_{1}}}+\frac{{{c}_{2}}}{{{r}_{2}}} \right)}{-\frac{{{t}^{-\beta }}}{\Gamma \left( 1-\beta  \right)}\left( \frac{{{c}_{1}}}{{{r}_{1}}}+\frac{{{c}_{2}}}{{{r}_{2}}} \right)}=\beta .
\end{equation}
\end{proof}
\begin{theorem}
\label{th3}
Let $1<\beta <2$, and $r$ be a real negative number. For the fractional differential equation with Caputo derivative
\begin{equation}
\label{113}
{{D}_{t}^{\beta }}u-ru=0,
\end{equation}
with the initial condition $u\left( 0 \right)=1$  and  ${u}'\left( 0 \right)=1$, the following relationship holds
\begin{equation}
\label{114}
\beta = 1 -\lim_{t\to\infty} \frac{t{u}'}{u}
\end{equation}
\end{theorem}
\begin{proof}
The exact solution to Eq. (\ref{113}) is
\begin{equation}\label{115}
u\left( t \right)={{E}_{\beta }}\left( r{{t}^{\beta }} \right)+t{{E}_{\beta ,2}}\left( r{{t}^{\beta }} \right).
\end{equation}
The first derivative of $u\left( t \right)$ could be calculated by referring to (\ref{12}) and (\ref{15})
\begin{equation}
\label{116}
{u}'\left( t \right)=r{{t}^{\beta -1}}{{E}_{\beta ,\beta }}\left( r{{t}^{\beta }} \right)+{{E}_{\beta }}\left( r{{t}^{\beta }} \right).
\end{equation}
The asymptotic expansions of $u\left( t \right)$ and ${u}'\left( t \right)$ at infinity are respectively
\begin{equation}
\label{117}
u\left( t \right)=-\frac{{{t}^{-\beta }}}{r\Gamma \left( 1-\beta  \right)}\left( 1+\frac{t}{1-\beta } \right)+O\left( {{\left| r \right|}^{-2}}{{t}^{-2\beta }} \right)+t O\left( {{\left| r \right|}^{-2}}{{t}^{-2\beta }} \right)
\end{equation}
and
\begin{equation}
\label{118}
{u}'\left( t \right)=\frac{{{t}^{-\beta }}}{r\Gamma \left( 1-\beta  \right)}\left( \frac{\beta }{t}-1 \right)+r{{t}^{\beta -1}}O\left( {{\left| {r} \right|}^{-3}}{{t}^{-3\beta }} \right)+ O\left( {{\left| {r} \right|}^{-2}}{{t}^{-2\beta }} \right)
\end{equation}
therefore
\begin{equation}
\label{119}
\frac{t{u}'}{u}=\frac{\frac{{{t}^{-\beta }}}{r\Gamma \left( 1-\beta  \right)}\left( \beta -t \right)+r{{t}^{\beta }}O\left( {{\left| {r} \right|}^{-3}}{{t}^{-3\beta }} \right)+O\left( {{\left| {r} \right|}^{-2}}{{t}^{-2\beta }} \right)}{-\frac{{{t}^{-\beta }}}{r\Gamma \left( 1-\beta  \right)}\left( 1+\frac{t}{1-\beta } \right)+O\left( {{\left| r \right|}^{-2}}{{t}^{-2\beta }} \right)+tO\left( {{\left| r \right|}^{-2}}{{t}^{-2\beta }} \right)}.
\end{equation}
As $t\to \infty $, Eq. (\ref{119}) leads to
\begin{equation}
\label{120}
\lim_{t\to\infty}\frac{t{u}'}{u}= \lim_{t\to\infty}\frac{t{u}'}{u}\,\frac{\frac{{{t}^{-\beta }}}{r\Gamma \left( 1-\beta  \right)}\left( \beta -t \right)}{-\frac{{{t}^{-\beta }}}{r\Gamma \left( 1-\beta  \right)}\left( 1+\frac{t}{1-\beta } \right)}=1-\beta
\end{equation}
and proof is completed.
\end{proof}
\begin{theorem}
\label{th4}
Let $0<\beta <1$, and $r$ be a real negative number. For the fractional differential equation (\ref{113}) with the initial condition $u\left( 0 \right)=1$, the following relationship holds
\begin{equation}
\label{121}
\beta = -\lim_{t\to\infty} \frac{t{u}'}{u}
\end{equation}
\end{theorem}
\begin{proof}
The exact solution to Eq. (\ref{113}), with  $0<\beta <1$ is in the form \cite{lin2013laplace}
\begin{equation}
\label{122}
u\left( t \right)={{E}_{\beta }}\left( r{{t}^{\beta }} \right).
\end{equation}
The first derivative of $u\left( t \right)$ could be calculated by referring to (\ref{12})
\begin{equation}
\label{123}
{u}'\left( t \right)=r{{t}^{\beta -1}}{{E}_{\beta ,\beta }}\left( r{{t}^{\beta }} \right).
\end{equation}
The asymptotic expansions of $u\left( t \right)$ and ${u}'\left( t \right)$ are respectively
\begin{equation}
\label{124}
u\left( t \right)=-\frac{{{t}^{-\beta }}}{r\Gamma \left( 1-\beta  \right)}+ O\left( {{\left| r \right|}^{-2}}{{t}^{-2\beta }} \right),
\end{equation}
and
\begin{equation}
\label{125}
{u}'\left( t \right)=\frac{\beta {{t}^{-\beta -1}}}{r\Gamma \left( 1-\beta  \right)}+r{{t}^{\beta -1}}O\left( {{\left| {r} \right|}^{-3}}{{t}^{-3\beta }} \right).
\end{equation}
Therefore
\begin{equation}
\label{126}
\frac{t{u}'}{u}=\frac{\frac{\beta {{t}^{-\beta}}}{r\Gamma \left( 1-\beta  \right)}+r{{t}^{\beta -1}}O\left( {{\left| {r} \right|}^{-3}}{{t}^{-3\beta }} \right)}{-\frac{{{t}^{-\beta }}}{r\Gamma \left( 1-\beta  \right)}+O\left( {{\left| r \right|}^{-2}}{{t}^{-2\beta }} \right)}.
\end{equation}
As $t\to \infty $, Eq. (\ref{126}) results in
\begin{equation}
\label{127}
\lim_ {t\to \infty }\,\frac{t{u}'}{u}=\lim_ {t\to \infty }\,\frac{\frac{\beta {{t}^{-\beta }}}{r\Gamma \left( 1-\beta  \right)}}{-\frac{{{t}^{-\beta }}}{r\Gamma \left( 1-\beta  \right)}}=-\beta .
\end{equation}
\end{proof}
\section{Numerical investigation}
\begin{example}
Consider the initial value problem
\begin{equation}\label{35}
D_{t}^{2\beta }u+2D_{t}^{\beta }u+0.7u=0,\qquad t\ge 0,\qquad 0<\beta <\frac{1}{2}
\end{equation}
with the initial condition $u\left( 0 \right)=1$. The exact solution to (\ref{35}) has the form
\begin{equation}
\label{36}
u\left( t \right)={{c}_{1}}{{E}_{\beta }}\left( {{r}_{1}}{{t}^{\beta }} \right)+{{c}_{2}}{{E}_{\beta }}\left( {{r}_{2}}{{t}^{\beta }} \right),
\end{equation}
where ${{r}_{1}}=-0.4523$, ${{r}_{2}}=-1.5477$, ${{c}_{1}}=1.4129$, ${{c}_{2}}= -0.4129$.
Figure \ref{fig1} represents the graph of $-\frac{t{u}'}{u}$, which has been evaluated for several different values of $\beta$, by using the exact representation of ${u}'$ and $u$. It could be obviously seen that $-\frac{t{u}'}{u}$ tends asymptotically to $\beta $, as $t$ goes to infinity. Numerical results coincide exactly with the result of the Theorem \ref{th2} and the rate of the convergence of $-\frac{t{u}'}{u}$ is greatly influenced by the value of $\beta $.
\begin{figure}[h!]
\centering
\includegraphics[scale=0.7]{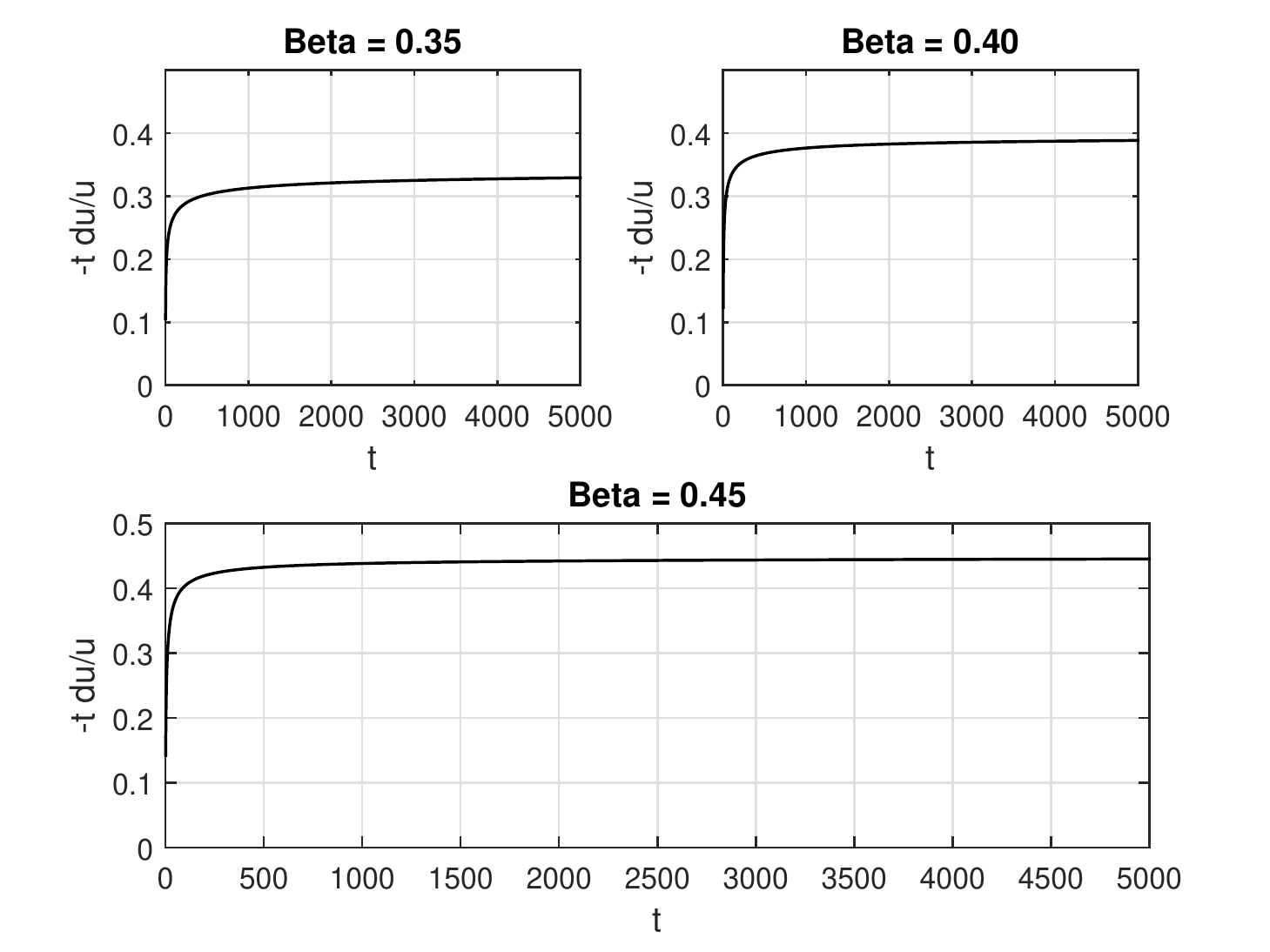}
\caption{Graph of $-\frac{t{u}'}{u}$ for $\beta =0.35$, $\beta =0.40$ and $\beta =0.45$.}
\label{fig1}
\end{figure}
\end{example}
\begin{example}
Consider the fractional differential equation
\begin{equation}
\begin{cases}
  {{D}_{t}^{\beta }}u+2u=0 \\
 u\left( 0 \right)=1 \\
 u'\left( 0 \right)=1
\end{cases}
\end{equation}
where $1<\beta <2$ and  ${{D}_{t}^{\beta }}u$ is in the sense of Caputo derivative. The exact solution is in the form of
\begin{equation}
u\left( t \right)={{E}_{\beta }}\left( -2{{t}^{\beta }} \right)+t{{E}_{\beta ,2}}\left( -2{{t}^{\beta }} \right).
\end{equation}
According to Theorem \ref{th3}, the term $1-\frac{t{u}'}{u}$ tends to the order $\beta $ as $t$ goes to the infinity. The numerical evaluation of $1-\frac{t{u}'}{u}$ has been conducted for different values of $\beta $, by using the exact expressions of ${u}'$ and $u$, shown in Figure \ref{fig2}. As it could be seen, $1-\frac{t{u}'}{u}$ converges to $\beta $ with a rate, which is obviously affected by the value of $\beta $, i.e. the convergence will be faster if the fractional order $\beta $ tends to 2.
\begin{figure}[h!]
\centering
\includegraphics[scale=0.7]{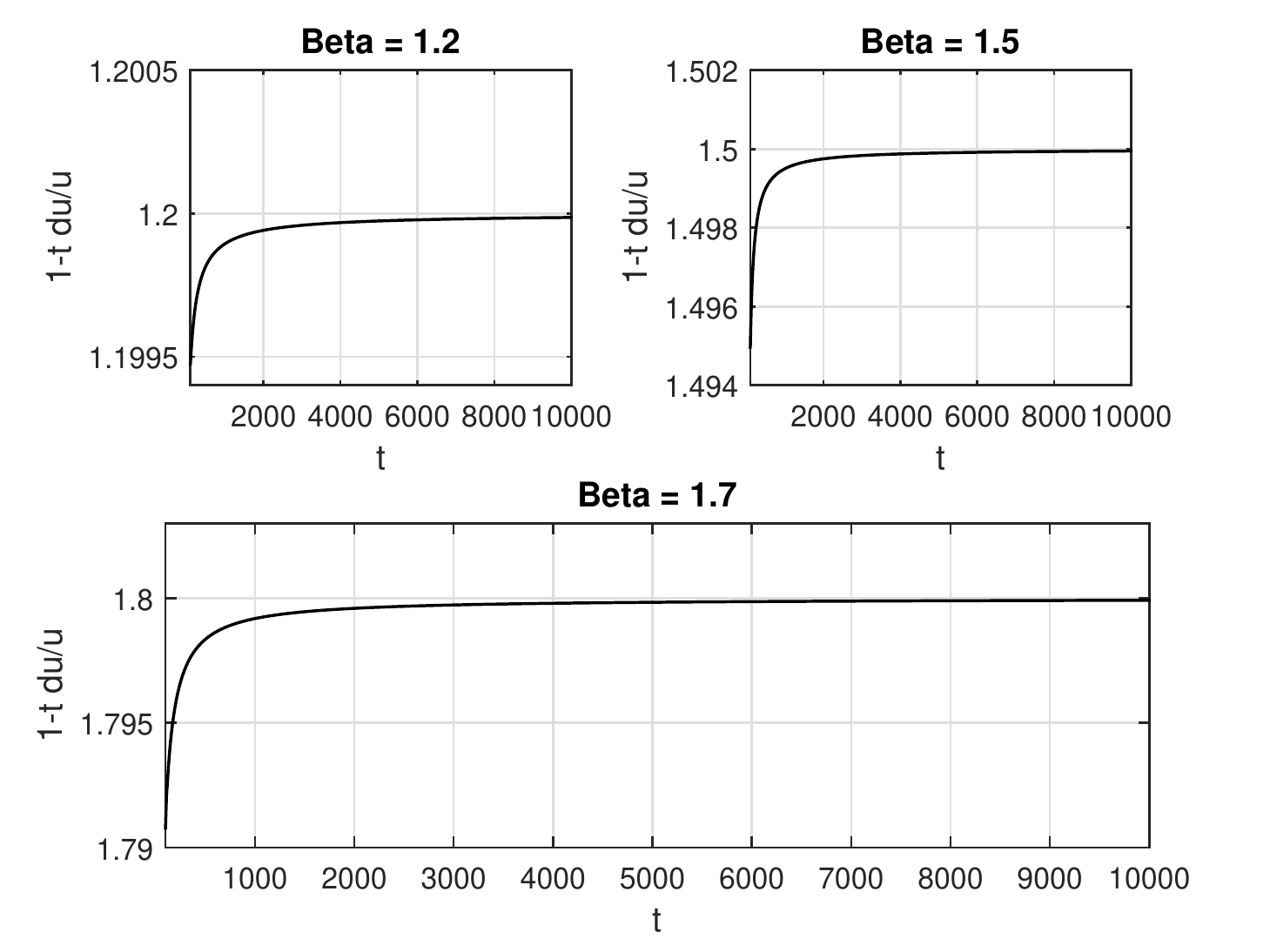}
\caption{Graph of $1-\frac{t{u}'}{u}$ for $\beta =1.2$, $\beta =1.5$ and $\beta =1.7$.}
\label{fig2}
\end{figure}
\end{example}
\section{Conclusion}
Inverse problem occurs in many branches of science and have been also examined in fractional differential systems. For instance, determination of the order of fractional systems has been indicated to be of such crucial importance that it could influence how anomalous diffusion equations must be appropriately simulated. Therefore, in this article, the exact solution to several classes of linear fractional differential equations represented, for which the fractional order determination was demonstrated by using asymptotic expansion of Mittag-Leffler function. Numerical tests have been conducted for different fractional orders, illustrating the accuracy of the formulae proved in the theorems.
\section*{Acknowledgements}
The authors are grateful to Professor Masahiro Yamamoto for pointing out the articles\cite{hatano2013determination,sakamoto2011initial,cheng2009uniqueness}.
\bibliographystyle{plain}
\bibliography{Asymp}
$^1$ Dipartimento di Scienze di Base e Applicate per l'Ingegneria \\
Sapienza Universit\`{a} di Roma \\
Via Antonio Scarpa n. 16 \\
00161 Rome, Italy \\[4pt]
  e-mail: mirko.dovidio@sbai.uniroma1.it
  \\[12pt]

  $^2$ Dipartimento di Scienze di Base e Applicate per l'Ingegneria \\
Sapienza Universit\`{a} di Roma \\
Via Antonio Scarpa n. 16 \\
00161 Rome, Italy \\[4pt]
  e-mail: paola.loreti@sbai.uniroma1.it
   \\[12pt]

    $^3$ Department of Structure and Material\\
Universiti Teknologi Malaysia\\
81310 Skudai, Johor, Malaysia\\[4pt]
e-mail: alireza.momenzadeh@gmail.com
   \\[12pt]

    $^4$ Dipartimento di Scienze di Base e Applicate per l'Ingegneria \\
Sapienza Universit\`{a} di Roma \\
Via Antonio Scarpa n. 16 \\
00161 Rome, Italy \\[4pt]
  e-mail: sima.sarvahrabi@sbai.uniroma1.it
  
\end{document}